\documentclass[psamsfonts,reqno]{amsart}
\usepackage{amssymb,eucal,graphics,latexsym}

\theoremstyle{plain}

\newtheorem{lemma}{Lemma}[section]

\newtheorem{proposition}[lemma]{Proposition}
\newtheorem{corollary}[lemma]{Corollary}

\newtheorem*{stat}{\name}
\newcommand{\name}{testing}

\theoremstyle{definition}
\newtheorem{definition}[lemma]{Definition}

\theoremstyle{remark}

\newtheorem{notation}[lemma]{Notation}

\newcommand{\qedc}{{\qed}~{\rm Claim~{\theclaim}.}}
\newcommand{\qedsc}{{\qed}~{\rm Claim.}}

\newcommand{\case}[1]%
{\smallskip\noindent\textbf{\textit{Case}\ {#1}.}}

\numberwithin{equation}{section}
\numberwithin{figure}{section}

\newcommand{\pup}[1]{\textup{(}{#1}\textup{)}}
\newcommand{\tvi}{\vrule height 12pt depth 6pt width 0pt}

\newcommand{\ga}{\alpha}
\newcommand{\gb}{\beta}
\newcommand{\gc}{\gamma}

\newcommand{\gf}{\varphi}

\newcommand{\gy}{\psi}
\newcommand{\gk}{\kappa}
\newcommand{\gl}{\lambda}

\newcommand{\gn}{\nu}

\newcommand{\gq}{\theta}
\newcommand{\gr}{\rho}

\newcommand{\gx}{\xi}
\newcommand{\gh}{\eta}

\newcommand{\go}{\omega}

\newcommand{\les}{\leqslant}
\newcommand{\intgr}[1]{\left\lfloor{#1}\right\rfloor}
\DeclareMathOperator{\kur}{kur}
\DeclareMathOperator{\brd}{br}
\DeclareMathOperator{\wdt}{wd}
\DeclareMathOperator{\card}{card}
\DeclareMathOperator{\cf}{cf}

\DeclareMathOperator{\J}{J}

\newcommand{\Pow}{\mathfrak{P}}

\newcommand{\tr}{\vartriangleleft}

\newcommand{\cD}{\mathcal{D}}

\newcommand{\cS}{\mathcal{S}}

\newcommand{\jirr}{join-ir\-re\-duc\-i\-ble}

\newcommand{\memb}{meet-em\-bed\-ding}

\newcommand{\es}{\varnothing}
\newcommand{\into}{\hookrightarrow}

\newcommand{\mono}{\rightarrowtail}

\newcommand{\famm}[2]{\left(#1\mid#2\right)}
\newcommand{\set}[1]{\{#1\}}
\newcommand{\setm}[2]{\set{#1\mid#2}}

\newcommand{\dnw}{\mathbin{\downarrow}}
\newcommand{\ddnw}{\mathbin{\downdownarrows}}

\newcommand{\js}{join-sem\-i\-lat\-tice}

\newcommand{\rB}{\mathrm{B}}

\begin{document}

\title{An infinite combinatorial statement with a poset parameter}

\author[P.~Gillibert]{Pierre Gillibert}
\author[F.~Wehrung]{Friedrich Wehrung}
\address{LMNO, CNRS UMR 6139\\
D\'epartement de Math\'ematiques, BP 5186\\
Universit\'e de Caen, Campus 2\\
14032 Caen cedex\\
France}
\email[P. Gillibert]{Pierre.Gillibert@math.unicaen.fr, pgillibert@yahoo.fr}
\urladdr[P. Gillibert]{http://www.math.unicaen.fr/\~{}giliberp/}

\email[F. Wehrung]{wehrung@math.unicaen.fr, fwehrung@yahoo.fr}
\urladdr[F. Wehrung]{http://www.math.unicaen.fr/\~{}wehrung}

\date{\today}



\thanks{This work was partially supported by the institutional grant MSM 0021620839}

\begin{abstract}
We introduce an extension, indexed by a partially ordered set~$P$ and cardinal numbers~$\gk$, $\gl$, denoted by $(\gk,{<}\gl)\leadsto P$, of the classical relation $(\gk,n,\gl)\rightarrow\nobreak\rho$ in infinite combinatorics. By definition, $(\gk,n,\gl)\rightarrow\gr$ holds if every map $F\colon[\gk]^n\to[\gk]^{<\gl}$ has a $\gr$-element free set. For example, Kuratowski's Free Set Theorem states that $(\gk,n,\gl)\rightarrow n+1$ holds if{f} $\gk\geq\gl^{+n}$, where $\gl^{+n}$ denotes the $n$-th cardinal successor of an infinite cardinal~$\gl$. By using the $(\gk,{<}\gl)\leadsto P$ framework, we present a self-contained proof of the first author's result that $(\gl^{+n},n,\gl)\rightarrow n+2$, for each infinite cardinal~$\gl$ and each positive integer~$n$, which solves a problem stated in the 1985 monograph of Erd\H{o}s, Hajnal, M\'at\'e, and Rado. Furthermore, by using an order-dimension estimate established in 1971 by Hajnal and Spencer, we prove the relation $(\gl^{+(n-1)},r,\gl)\rightarrow2^{\intgr{\frac{1}{2}(1-2^{-r})^{-n/r}}}$, for every infinite cardinal~$\gl$ and all positive integers~$n$ and~$r$ with $2\leq r<n$. For example, $(\aleph_{210},4,\aleph_0)\rightarrow 32{,}768$. Other order-dimension estimates yield relations such as $(\aleph_{109},4,\aleph_0)\rightarrow 257$ (using an estimate by F\"uredi and Kahn) and $(\aleph_7,4,\aleph_0)\rightarrow 10$ (using an exact estimate by Dushnik).
\end{abstract}

\maketitle

\section{Introduction}\label{S:Intro}
The present paper deals with finding large free sets for infinite set mappings of finite order. For a set~$X$ and a cardinal~$\gl$, we denote by $[X]^\gl$ (resp., $[X]^{<\gl}$) the set of all subsets of~$X$ with~$\gl$ elements (resp., less than~$\gl$ elements). For cardinals~$\gk$ and~$\gl$, a subset~$\cD$ of the powerset~$\Pow(\gk)$ of~$\gk$, and a map $F\colon\cD\to[\gk]^{<\gl}$, we say that a subset~$H$ of~$\gk$ is \emph{free with respect to~$F$} if $F(X)\cap H\subseteq X$ for each $X\in\Pow(H)\cap\cD$. For cardinals $\gk$, $\gl$, $\gr$, and for a positive integer~$r$, the statement $(\gk,r,\gl)\to\gr$ holds if every map $F\colon[\gk]^r\to[\gk]^{<\gl}$ (we say a \emph{set mapping of order~$r$}) has a $\gr$-element free set. Likewise, $(\gk,{<}\go,\gl)\to\gr$ means that every map $F\colon[\gk]^{<\go}\to[\gk]^{<\gl}$ has a $\gr$-element free set. Denote by $\gl^{+n}$ the $n$-th successor of an infinite cardinal~$\gl$. Kuratowski's Free Set Theorem (see~\cite{Kura51} or \cite[Theorem~46.1]{EHMR}) states that $(\gk,n,\gl)\to n+1$ if{f} $\gk\geq\gl^{+n}$, for all infinite cardinals~$\gk$ and~$\gl$ and every non-negative integer~$n$.

Whether larger free subsets can be found leads to unexpected discoveries. An argument that originates in L\'az\'ar~\cite{Laza36} shows that the relation $(\gl^+,1,\gl)\rightarrow\gl^+$ holds for each infinite cardinal~$\gl$ (cf. \cite[Corollary~44.2]{EHMR}). Therefore, the relation $(\gl^+,1,\gl)\rightarrow m$ holds \emph{a fortiori}, for each positive integer~$m$. Further, Hajnal-M\'at\'e~\cite{HaMa75} and Hajnal~\cite{EHMR}, respectively, proved that the relation $(\gl^{+r},r,\gl)\to m$ holds for every infinite cardinal~$\gl$, every $r\in\set{2,3}$, and every integer $m>r$ (cf. \cite[Theorem~46.2]{EHMR}).

In the presence of the Generalized Continuum Hypothesis~$\mathsf{GCH}$, we can say more. Indeed, it follows from \cite[Theorem~45.5]{EHMR} that the relation $(\gl^{+r},r,\gl)\to\gl^+$ holds for every integer~$r\geq2$ and every infinite cardinal~$\gl$.

On the other hand, without assuming~$\mathsf{GCH}$, the situation gains a touch of strangeness. Set $t_0:=5$, $t_1:=7$, and for each positive integer~$n$, $t_{n+1}$ is the least positive integer such that $t_{n+1}\rightarrow(t_n,7)^5$ (the latter notation meaning that for each $f\colon[t_{n+1}]^5\to\set{0,1}$, either there exists a $t_n$-element subset~$X$ of~$t_{n+1}$ such that $f``[X]^5=\set{0}$ or there exists a $7$-element subset~$X$ of~$t_{n+1}$ such that $f``[X]^5=\set{1}$). The existence of the sequence $\famm{t_n}{n<\go}$ is ensured by Ramsey's Theorem. It is established in Komj\'ath and Shelah~\cite{KoSh00} that for each positive integer~$n$, there exists a generic extension of the universe in which $(\aleph_n,4,\aleph_0)\not\rightarrow t_n$. In particular, there exists a generic extension of the universe in which $(\aleph_4,4,\aleph_0)\not\rightarrow t_4$.

By using algebraic tools of completely different nature and purpose introduced in~\cite{GillTh,Gill1} called \emph{compatible norm-coverings}, the first author established in \cite[Th\'e\-o\-r\`e\-me~3.3.13]{GillTh} the relation
 \begin{equation}\label{Eq:glrr+2}
 (\gl^{+r},r,\gl)\to r+2
 \end{equation}
for each positive integer~$r$ and each infinite cardinal~$\gl$, thus improving by one the cardinality of the free set given by Kuratowski's Free Set Theorem, and thus solving in the affirmative the question, raised in \cite[page~285]{EHMR}, whether $(\aleph_4,4,\aleph_0)\to6$. In the present paper, we develop a self-contained approach to such questions, and we extend the methods in order to find further large free sets results. In particular, we establish the relation
 \begin{equation}\label{Eq:ComplBound}
 (\gl^{+(n-1)},r,\gl)\rightarrow2^{\intgr{\frac{1}{2}(1-2^{-r})^{-n/r}}}\,,
 \end{equation}
for each infinite cardinal~$\gl$ and for all positive integers~$n$ and~$r$ with $2\leq r<n$; here $\intgr x$ denotes the largest integer below any real number~$x$.

Our main idea is to extend the $(\gk,r,\gl)\to\gr$ notation by introducing a partially ordered set (poset) parameter, thus defining the notation $(\gk,{<}\gl)\leadsto P$, for cardinals~$\gk$, $\gl$ and a poset~$P$, see Definition~\ref{D:InfCombP}. In particular, the statements $(\gk,{<}\gl)\leadsto([\gr]^{<\go},\subseteq)$ and $(\gk,{<}\go,\gl)\to\gr$ are equivalent, for all cardinals $\gk$, $\gl$, and~$\gr$ (cf. Proposition~\ref{P:TwoArrEquiv}). Then we define the \emph{Kuratowski index} of a finite poset~$P$ as the least non-negative integer~$n$ such that $(\gk^{+(n-1)},{<}\gk)\leadsto P$ holds for each infinite cardinal~$\gk$ (cf. Definition~\ref{D:KurInd}), with a minor adjustment for antichains. This definition is tailored in order to ensure that the finite Boolean lattice~$\Pow(n)$ of all subsets of~$n$ has Kuratowski index~$n$, for each positive integer~$n$. Smaller posets have smaller Kuratowski index (Proposition~\ref{P:KurIndIncr}) while the Kuratowski index function is subadditive on finite direct products (cf. Proposition~\ref{P:AddKur}). As a corollary, the Kuratowski index is bounded above by the order-dimension (cf. Proposition~\ref{P:UppBd}).

We apply these results to the \emph{truncated cubes} $\rB_m({\les}r)$ (cf. Notation~\ref{Not:B(k,m)}). In particular, we present the first author's proof of the relation~\eqref{Eq:glrr+2}, which uses the fact that the order-dimension of~$\rB_{n+2}({\les}n)$ (and thus also its Kuratowski index) is equal to~$n+1$. By using a further estimate of the order-dimension established in a 1971 paper by Spencer~\cite{Spen71}, we deduce the relation~\eqref{Eq:ComplBound} (cf. Proposition~\ref{P:LargeFreeSets}). For example, $(\aleph_{210},4,\aleph_0)\rightarrow 32{,}768$.

Another estimate of the order-dimension of~$\rB_m({\les}r)$, originating from a 1986 paper by F\"uredi and Kahn~\cite{FuKa86}, not so good asymptotically but giving larger free sets for smaller alephs (cf. Proposition~\ref{P:FuKa}), yields, for example, the relation
 \[
 (\aleph_{109},4,\aleph_0)\rightarrow 257\,.
 \]
Finally, Dushnik's original exact estimate yields, for example, the relation
 \[
 (\aleph_7,4,\aleph_0)\rightarrow 10\,.
 \]
 
\section{Basic concepts}\label{S:Basic}
\subsection{Set theory}\label{Su:BasicSet}
We shall use basic set-theoretical notation and terminology about ordinals and cardinals. 
We denote by~$f``(X)$, or~$f``X$ (resp., $f^{-1}X$) the image (resp., inverse image) of a set~$X$ under~$f$. Cardinals are initial ordinals. We denote by~$\cf\ga$ the cofinality of an ordinal~$\ga$. We denote by $\go:=\set{0,1,2,\dots}$ the first limit ordinal, mostly denoted by $\aleph_0$ in case it is viewed as a cardinal. We identify every non-negative integer~$n$ with the finite set $\set{0,1,\dots,n-1}$ (so $0=\es$). We denote by~$\gk^+$ the successor cardinal of a cardinal~$\gk$, and we define~$\gk^{+n}$, for a non-negative integer~$n$, by $\gk^{+0}:=\gk$ and $\gk^{+(n+1)}=(\gk^{+n})^+$.

We denote by~$\Pow(X)$ the powerset of a set~$X$, and we set
 \begin{align*}
 [X]^\gk&:=\setm{Y\in\Pow(X)}{\card Y=\gk}\,,\\
 [X]^{<\gk}&:=\setm{Y\in\Pow(X)}{\card Y<\gk}\,,\\
 [X]^{{\les}\gk}&:=\setm{Y\in\Pow(X)}{\card Y\leq\gk}\,,
 \end{align*}
for every cardinal~$\gk$.

\subsection{Partially ordered sets (posets)}\label{Su:Posets}
All our posets will be nonempty. For posets~$P$ and~$Q$, a map $f\colon P\to Q$ is \emph{isotone} if $x\leq y$ implies that $f(x)\leq f(y)$, for all $x,y\in P$.

We denote by~$0_P$ the least element of~$P$ if it exists. An element~$p$ in a poset~$P$ is \emph{\jirr} if $p=\bigvee X$ implies that $p\in X$, for every (possibly empty) finite subset~$X$ of~$P$; we denote by $\J(P)$ the set of all \jirr\ elements of~$P$, endowed with the induced partial ordering. We set
 \begin{align*}
 Q\dnw X&:=\setm{q\in Q}{(\exists x\in X)(q\leq x)}\,,\\
 Q\ddnw X&:=\setm{q\in Q}{(\exists x\in X)(q<x)}\,,
  \end{align*}
for all subsets~$Q$ and~$X$ of~$P$; in case $X=\set{a}$ is a singleton, then we shall write $Q\dnw a$ (resp., $Q\ddnw a$) instead of~$Q\dnw\set{a}$ (resp., $Q\ddnw\set{a}$). We set $\J_P(a):=\J(P)\dnw a$, for each $a\in P$. A subset~$Q$ of~$P$ is a \emph{lower subset of~$P$} if~$P\dnw Q=Q$. We say that~$P$ is a \emph{tree} if~$P$ has a smallest element and~$P\dnw a$ is a chain for each $a\in P$.

\section{An infinite combinatorial statement with a poset parameter}\label{L:leadsto}

The statement referred to in the section title is the following.

\begin{definition}\label{D:InfCombP}
For cardinals $\gk$, $\gl$ and a poset~$P$, let $(\gk,{<}\gl)\leadsto P$ hold if for every mapping $F\colon\Pow(\gk)\to[\gk]^{<\gl}$, there exists a one-to-one map $f\colon P\mono\nobreak\gk$ such that
 \begin{equation}\label{Eq:FreenuJPa}
 F(f``(P\dnw x))\cap f``(P\dnw y)\subseteq f``(P\dnw x)\,,
 \qquad\text{for all }x\leq y\text{ in }P\,.
 \end{equation}
\end{definition}

In most cases, we shall use the symbol $(\gk,{<}\gl)\leadsto P$ only in case~$P$ is \emph{lower finite}, in which case it is sufficient to assume that the mapping~$F$ is defined only on~$[\gk]^{<\go}$ (we can always set $F(X):=\es$ for infinite~$X$). Once this is set, it is of course sufficient to assume that~$F$ is \emph{isotone} (replace~$F(X)$ by $\bigcup\famm{F(Y)}{Y\subseteq X}$).

\begin{lemma}\label{L:SubPosetleadsto}
Let $P$ and $Q$ be posets, with~$Q$ lower finite, and let $\gk$, $\gl$ be cardinals. If $(\gk,{<}\gl)\leadsto Q$ and~$P$ embeds into~$Q$, then $(\gk,{<}\gl)\leadsto P$.
\end{lemma}

\begin{proof}
We may assume that~$P$ is a sub-poset of~$Q$. Let $F\colon[\gk]^{<\go}\to[\gk]^{<\gl}$ be isotone. By assumption, there exists a one-to-one map $g\colon Q\mono\nobreak\gk$ such that
 \[
 F(g``(Q\dnw x))\cap g``(Q\dnw y)\subseteq g``(Q\dnw x)\,,
 \qquad\text{for all }x\leq y\text{ in }Q\,.
 \]
The restriction~$f$ of~$g$ to~$P$ is one-to-one. As~$F$ is isotone and $f``(P\dnw x)\subseteq g``(Q\dnw x)$ for each $x\in P$, \eqref{Eq:FreenuJPa} is obviously satisfied.
\end{proof}

The following lemma states that the $(\gk,{<}\gl)\leadsto P$ relation can be ``verified on the \jirr\ elements''.

\begin{lemma}\label{L:EquivKurat}
Let $P$ be a lower finite poset and let~$\gl$ and~$\gk$ be infinite cardinals. Then $(\gk,{<}\gl)\leadsto P$ if{f} for every isotone mapping $G\colon[\gk]^{<\go}\to[\gk]^{<\gl}$, there exists a one-to-one map $g\colon\J(P)\mono\nobreak\gk$ such that
 \begin{equation}\label{Eq:FreenuJPa2}
 G(g``\J_P(x))\cap g``\J_P(y)\subseteq g``\J_P(x)\,,
 \qquad\text{for all }x\leq y\text{ in }P\,.
 \end{equation}
\end{lemma}

\begin{proof}
It is trivial that $(\gk,{<}\gl)\leadsto P$ entails the given condition. For the converse, we shall fix a bijection $\gf\colon[\gk]^{<\go}\to\gk$, and set
 \[
 \Phi(X):=\bigcup\gf^{-1}X\,,\quad\text{for each }X\in\Pow(\gk)\,.
 \]
That is, $\Phi(X)$ is the union of all finite subsets~$Y$ of~$\gk$ such that $\gf(Y)\in X$. In particular, $\Phi(X)$ belongs to $[\gk]^{<\gl}$ whenever $X$ belongs to $[\gk]^{<\gl}$.
Now let $F\colon[\gk]^{<\go}\to[\gk]^{<\gl}$ be an isotone map. We can define a map $G\colon[\gk]^{<\go}\to[\gk]^{<\gl}$ by the rule
 \[
 G(X):=\Phi\bigl(F\bigl(\setm{\gf(Y)}{Y\subseteq X}\bigr)\bigr)\,,\quad
 \text{for each }X\in[\gk]^{<\go}\,.
 \]
By assumption, there exists a one-to-one map $g\colon\J(P)\mono\gk$ which satisfies the condition~\eqref{Eq:FreenuJPa2}. As~$P$ is lower finite, we can define a map $f\colon P\to\gk$ by the rule
 \[
 f(x):=\gf(g``\J_P(x))\,,\quad\text{for each }x\in P\,.
 \]
As~$P$ is lower finite, every element of~$P$ is the join of all elements of~$\J(P)$ below it. Hence, as both~$g$ and~$\gf$ are one-to-one, $f$ is one-to-one as well.

Now we prove that the condition~\eqref{Eq:FreenuJPa} holds. Let $x\leq y$ in~$P$ and let $u\in P\dnw y$ such that $f(u)\in F(f``(P\dnw x))$, we must prove that $u\leq x$. As
 \[
 f``(P\dnw x)=\setm{\gf(g``\J_P(v))}{v\in P\dnw x}
 \subseteq\setm{\gf(Y)}{Y\subseteq g``\J_P(x)}
 \]
and~$F$ is isotone, we get
 \[
 \gf(g``\J_P(u))=f(u)\in F(f``(P\dnw x))\subseteq
 F(\setm{\gf(Y)}{Y\subseteq g``\J_P(x)})\,,
 \]
thus $g``\J_P(u)\subseteq\Phi\bigl(F(\setm{\gf(Y)}{Y\in g``\J_P(x)})\bigr)=G(g``\J_P(x))$. This means that for each $p\in\J_P(u)$, $g(p)$ belongs to $G(g``\J_P(x))$. As~$g(p)$ also belongs to $g``\J_P(y)$ (because $p\leq u\leq y$), it follows, using~\eqref{Eq:FreenuJPa2}, that $g(p)\in g``\J_P(x)$, and so $p\leq x$. As this holds for each $p\in\J_P(u)$, we obtain that $u\leq x$, as was to be proved.
\end{proof}

We shall now recall a few known facts about the relation $(\gk,{<}\go,\gl)\to\nobreak\gr$ (cf. Section~\ref{S:Intro}). In case $\gl\geq\aleph_1$ and~$\gr\geq\aleph_0$, the existence of~$\gk$ such that $(\gk,{<}\go,\gl)\to\gr$ is a large cardinal axiom, that entails the existence of~$0^{\#}$ in case~$\gr\geq\aleph_1$ (cf.~\cite{DePa}, and also \cite{Koep84,Koep89} for further related consistency strength results). The relation $(\gk,{<}\go,\gl)\to\gr$ follows from the infinite partition property $\gk\to(\gq)^{<\go}_2$ (\emph{existence of the $\gq^{\mathrm{th}}$ Erd\H{o}s cardinal}) where $\gq:=\max\set{\gr,\gl^+}$, see \cite[Theorem~45.2]{EHMR} and the discussion preceding it. This notation is related to ours \emph{via} the following easy result.

\begin{proposition}\label{P:TwoArrEquiv}
The statements $(\gk,{<}\gl)\leadsto([\gr]^{<\go},\subseteq)$ and $(\gk,{<}\go,\gl)\to\gr$ are equivalent, for every cardinal~$\gr$ and all infinite cardinals $\gk$, $\gl$.
\end{proposition}

\begin{proof}
We use the characterization of the $\leadsto$ relation given by Lemma~\ref{L:EquivKurat}.
First observe that $\J([\gr]^{<\go})=\setm{\set{\gx}}{\gx<\gr}$. Assume that the statement $(\gk,{<}\gl)\leadsto[\gr]^{<\go}$ holds, and let $F\colon[\gk]^{<\go}\to[\gk]^{<\gl}$. By using the easy direction of Lemma~\ref{L:EquivKurat}, we obtain a one-to-one map $f\colon\gr\mono\gk$ such that, putting $H:=f``(\gr)$,
 \begin{equation}\label{Eq:VarFreenuJPa}
 F(f``(p))\cap H\subseteq f``(p)\,,\qquad\text{for each }p\in[\gr]^{<\go}\,.
 \end{equation}
This means that $H$ is free with respect to~$F$.

Conversely, assume that $(\gk,{<}\go,\gl)\to\gr$ holds, and let $F\colon\Pow(\gk)\to[\gk]^{<\gl}$. By assumption, there exists $H\in[\gk]^\gr$ which is free with respect to~$F$. Then \eqref{Eq:VarFreenuJPa} holds, for any one-to-one map~$f\colon\gr\mono H$. Hence $(\gk,{<}\gl)\leadsto([\gr]^{<\go},\subseteq)$ holds.
\end{proof}

The following observation illustrates the difficulties of checking Definition~\ref{D:InfCombP} on arbitrary posets. 

\begin{proposition}\label{P:NoFreeNonLF}
There exists no cardinal~$\gk$ satisfying $(\gk,{<}\aleph_1)\leadsto(\Pow(\go),\subseteq)$.
\end{proposition}

\begin{proof}
Observe that $\J(\Pow(\go))=\setm{\set{n}}{n<\go}$.
Let $X\equiv_{\mathrm{fin}}Y$ hold if the symmetric difference~$X\mathbin{\triangle}Y$ is finite, for all subsets~$X$ and~$Y$ of~$\gk$. Let~$\Delta$ be a set that meets the $\equiv_{\mathrm{fin}}$-equivalence class~$[X]_{\equiv_{\mathrm{fin}}}$ of~$X$ in exactly one point, for each~$X\in\Pow(\gk)$, and denote by~$F(X)$ the unique element of~$[X]_{\equiv_{\mathrm{fin}}}\cap\Delta$ if~$X$ is at most countable, $\es$ otherwise. If the relation $(\gk,{<}\aleph_1)\leadsto(\Pow(\go),\subseteq)$ holds, then there is a one-to-one map $f\colon\go\mono\gk$ such that $F(f``(p))\cap f``(q)\subseteq f``(p)$ for all~$p\subseteq q$ in~$\Pow(\go)$. As $F(f``(\go))\equiv_{\mathrm{fin}}f``(\go)$, there exists $m<\go$ such that $f``(\go\setminus m)\subseteq F(f``(\go))$. Put $p:=\go\setminus(m+1)$. It follows from the relation $f``(p)\equiv_{\mathrm{fin}}f``(\go)$ that $F(f``(p))=F(f``(\go))$, thus
 \[
 f(m)\in F(f``(\go))\cap f``(\go)=F(f``(p))\cap f``(\go)\subseteq f``(p)\,,
 \]
while $f(m)\notin f``(p)$, a contradiction.
\end{proof}

\section[Kuratowski index]{The Kuratowski index of a finite poset}\label{S:KuratInd}

\begin{definition}\label{D:KurInd}
A non-negative integer~$n$ is \emph{\textbf{a} Kuratowski index} of a finite poset~$P$ if either~$P$ is an antichain, or~$P$ is not an antichain, $n>0$, and the relation $(\gk^{+(n-1)},{<}\gk)\leadsto P$ holds for each infinite cardinal~$\gk$. Furthermore, we shall call \emph{\textbf{the} Kuratowski index of~$P$} the least such~$n$ if it exists (we shall see shortly that it does), and we shall denote it by~$\kur(P)$.
\end{definition}

The reason of our adjustment of Definition~\ref{D:KurInd} for antichains lies essentially in the fact that the relation $(\gk,{<}\gl)\leadsto P$ trivially holds for every antichain~$P$ and all cardinals~$\gk$, $\gl$ such that $\gk\geq\card P$.

\begin{proposition}\label{P:Existskur(P)}
The number of \jirr\ elements in~$P$ is a Kuratowski index of~$P$, for each finite poset~$P$. Hence, $\kur(P)$ is defined and $\kur(P)\leq\card\J(P)$.
\end{proposition}

\begin{proof}
The statement is trivial in case~$P$ is an antichain, so suppose that~$P$ is not an antichain. Set $n:=\card\J(P)$, let $\gl$ be an infinite cardinal, and set $\gk:=\gl^{+(n-1)}$. It follows from Kuratowski's Free Set Theorem~\cite{Kura51} that the relation $(\gk,{<}\gl)\leadsto(\Pow(n),\subseteq)$ holds (cf. Proposition~\ref{P:TwoArrEquiv}). As the assignment $x\mapsto\J_P(x)$ defines an embedding from~$P$ into $(\Pow(\J(P)),\subseteq)$, it follows from Lemma~\ref{L:SubPosetleadsto} that the relation $(\gk,{<}\gl)\leadsto P$ holds.
\end{proof}

\begin{proposition}\label{P:KuratTree}
Let~$T$ be a well-founded tree and let~$\gk$ be an infinite cardinal such that $\card T\leq\gk$ and $\card(T\dnw p)<\cf\gk$ for each $p\in T$. Then the relation $(\gk,{<}\gk)\leadsto T$ holds.
\end{proposition}

\begin{proof}
As~$T$ is well-founded, there exists a unique map~$\gr$ from~$T$ to the ordinals (the rank map) such that $\gr(q)=\setm{\gr(p)}{p\in T\ddnw q}$ for each $q\in T$. The range~$\gq$ of the map~$\gr$ is an ordinal and $\card\gq\leq\gk$.
Fix a partition $\famm{I_\gh}{\gh<\gq}$ of~$\gk$ such that $\card I_\gh=\gk$ for each $\gh<\gq$. Fix also a one-to-one map $\gn\colon T\mono\gk$. We define a strict ordering~$\tr$ on~$T$ (lexicographical product) by
 \[
 p\tr q\ \Longleftrightarrow\ \bigl(\text{either }\gr(p)<\gr(q)\text{ or }
 (\gr(p)=\gr(q)\text{ and }\gn(p)<\gn(q))\bigr)\,.
 \]
Observe that~$\tr$ is a strict well-ordering of~$T$ and that it extends the original strict ordering~$<$. Let $F\colon\Pow(\gk)\to[\gk]^{<\gk}$. We define, by $\tr$-induction, a map $f\colon T\to\gk$, as follows. Let $q\in T$ and suppose that $f(p)$ has been defined and belongs to~$I_{\gr(p)}$, for each $p\in T$ such that $p\tr q$. It follows from the assumptions on~$T$ that the set
 \[
 X_q:=\setm{f(p)}{p\in T\,,\ \gr(p)=\gr(q)\,,\ \gn(p)<\gn(q)}\cup
 \bigcup\famm{F(f``(T\dnw p))}{p\in T\ddnw q}
 \]
(where $T\ddnw q$ is evaluated with respect to the original ordering of~$T$) has cardinality less than~$\gk$. Hence, $I_{\gr(q)}\setminus X_q$ is nonempty. We pick $f(q)\in I_{\gr(q)}\setminus X_q$.

We claim that~$f$ is one-to-one. Indeed let $p,q\in T$ distinct, we prove that $f(p)\neq f(q)$. We may assume that $p\tr q$. If $\gr(p)=\gr(q)$, then $\gn(p)<\gn(q)$, so $f(p)\in X_q$ and the conclusion follows. If $\gr(p)<\gr(q)$, then the conclusion follows from $f(p)\in I_{\gr(p)}$, $f(q)\in I_{\gr(q)}$, and $I_{\gr(p)}\cap I_{\gr(q)}=\es$.

Furthermore, for all $p<q$ in~$T$, $f(q)$ does not belong to~$X_q$, thus, \emph{a fortiori}, not to~$F(f``(T\dnw p))$. Therefore, $F(f``(T\dnw p))\cap f``(T\dnw q)\subseteq f``(T\dnw p)$ for all $p\leq q$ in~$T$.
\end{proof}

{}From Lemma~\ref{L:SubPosetleadsto} we deduce immediately the following.

\begin{proposition}\label{P:KurIndIncr}
If a poset~$P$ embeds into a finite poset~$Q$, then $\kur(P)\leq\kur(Q)$.
\end{proposition}

The proof of the following result is inspired by the proof of Kuratowski's Free Set Theorem, see \cite[Theorem~46.1]{EHMR}. It is closely related to \cite[Lemme~3.3.7]{GillTh}.

\begin{lemma}[Product Lemma]\label{L:ProdKurat}
Let $P$ and $Q$ be lower finite posets with zero and let~$\ga$, $\gb$, $\gc$ be infinite cardinals such that $\ga\leq\gb\leq\gc$ and $\card Q<\cf\ga$.
If $(\gb,{<}\ga)\leadsto P$ and $(\gc,{<}\gb^+)\leadsto Q$, then $(\gc,{<}\ga)\leadsto P\times Q$.
\end{lemma}

\begin{proof}
We use the equivalent form of the~$\leadsto$ relation provided by Lemma~\ref{L:EquivKurat}.
Let $F\colon\Pow(\gc)\to[\gc]^{<\ga}$. As $\gb\leq\gc$, there exists a partition $(U,V)$ of the set~$\gc$ such that $\card U=\gb$ and $\card V=\gc$. As $\ga\leq\gb$, we can define a map $H\colon\Pow(V)\to[V]^{<\gb^+}$ by the rule
 \[
 H(Y):=V\cap\bigcup\famm{F(X\cup Y)}{X\in[U]^{<\go}}\,,\quad\text{for each }
 Y\in\Pow(V)\,.
 \]
As $(\gc,{<}\gb^+)\leadsto Q$, there exists a one-to-one map $h\colon\J(Q)\mono V$ such that
 \begin{equation}\label{Eq:hfreeHQ}
 (\forall y\leq y'\text{ in }Q)\bigl(H(h``\J_Q(y))\cap h``\J_Q(y')\subseteq
 h``\J_Q(y)\bigr)\,.
 \end{equation}
As $\card Q<\cf\ga$, we can define a map $G\colon\Pow(U)\to[U]^{<\ga}$ by the rule
 \[
 G(X):=U\cap\bigcup\famm{F(X\cup h``\J_Q(y))}{y\in Q}\,,\quad\text{for each }
 X\in\Pow(U)\,.
 \]
As $(\gb,{<}\ga)\leadsto P$, there exists a one-to-one map $g\colon\J(P)\mono U$ such that
 \begin{equation}\label{Eq:gfreeGP}
 (\forall x\leq x'\text{ in }P)\bigl(G(g``\J_P(x))\cap g``\J_P(x')\subseteq
 g``\J_P(x)\bigr)\,.
 \end{equation}
As $\J(P\times Q)=(\J(P)\times\set{0_Q})\cup(\set{0_P}\times\J(Q))$ and~$\gc$ is the disjoint union of~$U$ and~$V$, we can define a one-to-one map $f\colon\J(P\times Q)\mono\gc$ by the rule
 \[
 \begin{cases}
 f(p,0_Q)&:=g(p)\,,\quad\text{for each }p\in\J(P)\,;\\
 f(0_P,q)&:=h(q)\,,\quad\text{for each }q\in\J(Q)\,. 
 \end{cases}
 \]
Let $(x,y)\leq(x',y')$ in $P\times Q$, we verify that the following statement holds:
 \begin{equation}\label{Eq:ffreePtimesQ}
 F(f``\J_{P\times Q}(x,y))\cap f``\J_{P\times Q}(x',y')\subseteq
 f``\J_{P\times Q}(x,y)\,. 
 \end{equation}
So let $(p,q)\in\J_{P\times Q}(x',y')$ such that $f(p,q)\in F(f``\J_{P\times Q}(x,y))$, that is, $f(p,q)\in F(g``\J_P(x)\cup h``\J_Q(y))$. We must prove that $(p,q)\leq(x,y)$. We separate cases.

\case{1} $q=0_Q$. As
 \[
  g(p)=f(p,q)\in F(g``\J_P(x)\cup h``\J_Q(y))\cap U\subseteq G(g``\J_P(x))
 \]
while $g(p)\in g``\J_P(x')$, it follows from~\eqref{Eq:gfreeGP} that $g(p)\in g``\J_P(x)$, that is, $p\leq x$.

\case{2} $p=0_P$. As
 \[
 h(q)=f(p,q)\in F(g``\J_P(x)\cup h``\J_Q(y))\cap V\subseteq H(h``\J_Q(y))
 \]
while $h(q)\in h``\J_Q(y')$, it follows from~\eqref{Eq:hfreeHQ} that $h(q)\in h``\J_Q(y)$, that is, $q\leq y$.
This completes the proof of \eqref{Eq:ffreePtimesQ}.
\end{proof}

\begin{proposition}\label{P:AddKur}
The inequality $\kur(P\times Q)\leq\kur(P)+\kur(Q)$ holds, for any finite posets $P$ and~$Q$ with zero.
\end{proposition}

\begin{proof}
We may assume that~$P$ and~$Q$ are both nonzero. Set $m:=\kur(P)$ and $n:=\kur(Q)$. Let~$\gk$ be an infinite  cardinal and set $\gl:=\gk^{+m}$. As $(\gk^{+(m-1)},{<}\gk)\leadsto\nobreak P$ and $(\gl^{+(n-1)},{<}\gl)\leadsto Q$, it follows from the Product Lemma (Lemma~\ref{L:ProdKurat}) that $(\gl^{+(n-1)},{<}\gk)\leadsto P\times Q$. Now observe that $\gl^{+(n-1)}=\gk^{+(m+n-1)}$.
\end{proof}

Denote by $\dim(P)$ the \emph{order-dimension} of a poset~$P$, that is, the smallest number~$\gk$ of chains such that~$P$ embeds into a product of~$\gk$ chains. The order-dimension of a finite poset is finite. It follows immediately from Proposition~\ref{P:KuratTree} that the Kuratowski index of a nontrivial finite tree is~$1$. Hence, by applying Proposition~\ref{P:AddKur}, we obtain immediately the following upper bound for~$\kur(P)$.

\begin{proposition}\label{P:UppBd}
Le~$n$ be a positive integer and let~$P$ be a finite poset. If~$P$ embeds, as a poset, into a product of~$n$ trees, then $\kur(P)\leq n$. In particular, $\kur(P)\leq\dim(P)$.
\end{proposition}

As $\dim(P)\leq\card\J(P)$ (for the assignment $x\mapsto\J_P(x)$ defines a \memb\ from~$P$ into~$\Pow(\J(P))$), this bound is sharper than the previously observed bound $\card\J(P)$ (cf. Proposition~\ref{P:Existskur(P)}). In fact, it dates back to Baker~\cite{Baker} that $\dim(P)\leq\wdt\J(P)$, where~$\wdt$ denotes the \emph{width} function (the width of a poset is the supremum of the cardinalities of all antichains of that poset). Hence,
 \begin{equation}\label{Eq:VarIneqkurP}
 \kur(P)\leq\dim(P)\leq\wdt\J(P)\leq\card\J(P)\,.
 \end{equation}
While Proposition~\ref{P:UppBd} provides an upper bound for the Kuratowski index of a finite poset, our next result, Proposition~\ref{P:kurgewbr}, will provide a lower bound.
The classical definition of breadth, see Baker~\cite{Baker} or Ditor \cite[Section~4]{Dito84}, runs as follows. Let~$n$ be a positive integer. A \js\ $P$ has \emph{breadth at most~$n$} if for every nonempty finite subset~$X$ of~$P$, there exists a nonempty $Y\subseteq X$ with at most~$n$ elements such that $\bigvee X=\bigvee Y$. This is a particular case of the following definition of breadth, valid for every poset, and, in addition, \emph{self-dual}: we say that a poset~$P$ has breadth at most~$n$ if for all $x_i$, $y_i$ \pup{$0\leq i\leq n$} in $P$, if $x_i\leq y_j$ for all $i\neq j$ in $\set{0,1,\dots,n}$, then there exists $i\in\set{0,1,\dots,n}$ such that $x_i\leq y_i$. We denote by $\brd(P)$ the least non-negative integer~$n$, if it exists, such that~$P$ has breadth at most~$n$, and we call it \emph{\textbf{the} breadth of~$P$}.

\begin{proposition}\label{P:kurgewbr}
Let $P$ be a finite poset with zero such that $P\dnw a$ is a \js\ for each $a\in P$. Then the inequality $\max\setm{\brd(P\dnw a)}{a\in P}\leq\kur(P)$ holds.
\end{proposition}

\begin{proof}
We may assume that $P\neq\set{0}$.
Hence $n:=\max\setm{\brd(P\dnw a)}{a\in P}$ is a positive integer. Pick $a\in P$ such that $n=\brd(P\dnw a)$. There are $p_0,\dots,p_{n-1}\in P\dnw a$ such that $p_i\nleq\bigvee\famm{p_j}{j\neq i}$ for each $i<n$, where the join $\bigvee\famm{p_j}{j\neq i}$ is evaluated in~$P\dnw a$. The assignment $X\mapsto\bigvee\famm{p_i}{i\in X}$ defines a join-embedding from~$\Pow(n)$ into~$P$, thus, by Proposition~\ref{P:KurIndIncr}, $\kur(P)\geq\kur(\Pow(n))=n$ (the latter equality is part of the content of Kuratowski's Free Set Theorem, cf. Proposition~\ref{P:TwoArrEquiv}).
\end{proof}

As an immediate application of Propositions~\ref{P:UppBd} and~\ref{P:kurgewbr}, we obtain the following.

\begin{corollary}\label{C:kur(prdntr)}
Let $n$ be a positive integer and let~$P$ be a product of~$n$ nontrivial finite trees. Then $\kur(P)=n$.
\end{corollary}

\section{Kuratowski indexes of truncated cubes}\label{S:Trunc}

In this section, we shall give estimates of Kuratowski indexes of \emph{truncated cubes}, with applications to finding large free sets for set mappings of order greater than one.

\begin{notation}\label{Not:B(k,m)}
For integers~$r$, $k$, $r_0\leq\cdots\leq r_{k-1}$, and~$m$ such that $1\leq r\leq m$ and $1\leq r_j\leq m$ for each $j<k$, we define the \emph{truncated $m$-dimensional cubes}
 \begin{align*}
 \rB_m({\les}r)&:=\setm{X\in\Pow(m)}{\text{either }\card X\leq r\text{ or }X=m}\,,\\
 \rB_m(r_0,\dots,r_{k-1})&:=
 \setm{X\in\Pow(m)}{\card X\in\set{r_0,\dots,r_{k-1}}}\,.
 \end{align*}
endowed with containment.
\end{notation}

Diagrams indexed by~$\rB_m({\les}2)$, for $m>2$, are widely used in Gillibert~\cite{Gill2}.

As $\rB_m({\les}m)=\rB_m({\les}m-1)$ for each positive integer~$m$, we shall assume that $r<m$ whenever we consider the poset~$\rB_m({\les}r)$. The following result is related to \cite[Corollaire~3.3.3]{GillTh}.

\begin{proposition}\label{P:FreeB(k,m)}
Let $r$ and $m$ be integers with $1\leq r<m$ and let $\gk$ and $\gl$ be infinite cardinals. Then the following statements are equivalent:
\begin{enumerate}
\item $(\gk,{<}\gl)\leadsto\rB_m({\les}r)$;

\item for every $F\colon[\gk]^{{\les}r}\to[\gk]^{<\gl}$ there exists $H\in[\gk]^m$ such that $F(X)\cap H\subseteq X$ for each $X\in[H]^{{\les}r}$;

\item $(\gk,r,\gl)\to m$ \pup{cf. Section~\textup{\ref{S:Intro}}}.
\end{enumerate}
\end{proposition}

\begin{proof}
An easy argument, similar to the one used in the proof of Proposition~\ref{P:TwoArrEquiv}, yields the equivalence of~(i) and~(ii). Furthermore, as every map from~$[\gk]^r$ to~$[\gk]^{<\gl}$ trivially extends to a map from~$[\gk]^{{\les}r}$ to~$[\gk]^{<\gl}$, (ii) implies~(iii).

Finally assume that~(iii) holds, and let $F\colon[\gk]^{{\les}r}\to[\gk]^{<\gl}$. We must find an~$m$-element free set, with respect to~$F$, of~$\gk$. We may assume that~$F$ is isotone. By applying~(ii) to the restriction~$F'$ of~$F$ to~$[\gk]^r$, we obtain an~$m$-element subset~$H$ of~$\gk$ which is free with respect to~$F'$. Let $X\in[H]^{{\les}r}$. For each $\gx\in H\setminus X$, from $X\subseteq H\setminus\set{\gx}$ and $\card X\leq r\leq\card(H\setminus\set{\gx})$ it follows that there exists $Y\in[H]^r$ such that $X\subseteq Y\subseteq H\setminus\set{\gx}$. By applying this to all elements $\gx\in H\setminus X$, we obtain that we can write $X=\bigcap\famm{X_i}{i<s}$, for a positive integer~$s$ and subsets $X_0$, \dots, $X_{s-1}$ of~$H$. As $F(X_i)\cap H\subseteq X_i$ for each~$i<s$ and as~$F$ is isotone, we obtain that $F(X)\cap H\subseteq X$. Therefore, $H$ is free with respect to~$F$, and so~(ii) holds.
\end{proof}

As~$\rB_m({\les}r)$ is a finite lattice with breadth~$r+1$, it follows from Proposition~\ref{P:kurgewbr} that $\kur\rB_m({\les}r)\geq r+1$, for all integers~$r$ and~$m>r$ such that $1\leq r<m$. Furthermore, as~$\rB_m({\les}r)$ has exactly~$m$ \jirr\ elements (namely the singletons), it follows from Proposition~\ref{P:Existskur(P)} that $\kur\rB_m({\les}r)\leq m$.

On the other hand, according to the results of L\'az\'ar, Hajnal-M\'at\'e, and Hajnal cited in Section~\ref{S:Intro}, it follows from Proposition~\ref{P:FreeB(k,m)} that the relation\linebreak $(\gl^{+r},{<}\gl)\leadsto\rB_m({\les}r)$ holds for all integers~$r\in\set{1,2,3}$ and~$m>r$. In particular, $\kur\rB_m({\les}r)\leq r+1$, and thus, as the converse inequality holds, $\kur\rB_m({\les}r)=r+1$, whenever $r\in\set{1,2,3}$ and $r<m<\go$. Arguing as above, in the presence of~$\mathsf{GCH}$ we obtain from the relation $(\gl^{+r},r,\gl)\to\gl^+$ (cf. \cite[Theorem~45.5]{EHMR}) that $\kur\rB_m({\les}r)=r+1$---now for all integers~$r$ and~$m$ such that $1\leq r<m$.

Without assuming~$\mathsf{GCH}$, Komj\'ath and Shelah's result (cf. Section~\ref{S:Intro}) yields that $\kur\rB_{t_4}({\les}4)$ may be larger than or equal to~$6$. In particular, $\kur\rB_{t_4}({\les}4)=5$ in any set-theoretical universe satisfying~$\mathsf{GCH}$, while $\kur\rB_{t_4}({\les}4)\geq6$ in some generic extension. And therefore, the \emph{Kuratowski index function is not absolute} (in the set-theoretical sense).

We sum up in the following proposition some of the results above.

\begin{proposition}\label{P:MoreKurB(r,m)}\hfill
\begin{enumerate}
\item $r+1\leq\kur\rB_m({\les}r)\leq m$ for all integers~$r$ and~$m$ such that $1\leq r<m$.

\item $\kur\rB_m({\les}r)=r+1$ for all integers~$r$ and~$m$ such that $1\leq r<m$ and $r\in\set{1,2,3}$.

\item Assume that~$\mathsf{GCH}$ holds. Then $\kur\rB_m({\les}r)=r+1$ for all integers~$r$ and~$m$ such that $1\leq r<m$.

\item There exists a model of $\mathsf{ZFC}$ where $\kur\rB_{t_4}({\les}4)\geq6$.
\end{enumerate}
\end{proposition}

While the integer~$t_4$ of Proposition~\ref{P:MoreKurB(r,m)} is quite large, smaller values yield new Kuratowski indexes that were not available by using earlier methods. The following result is due to the first author \cite[Lemme~3.3.12]{GillTh}. The proof of the order-dimension exact estimate $\dim\rB_{n+2}({\les}n)=n+1$ is contained in~\cite{Dush50}, however \cite[Lemme~3.3.12]{GillTh} provides an easy direct proof in that case. Of course it is obvious that the breadth of~$\rB_{n+2}({\les}n)$ is $n+1$, and so the conclusion of Lemma~\ref{L:OrdDImBrm} follows from Propositions~\ref{P:UppBd} and~\ref{P:kurgewbr}.

\begin{lemma}\label{L:OrdDImBrm}
$\kur\rB_{n+2}({\les}n)=\dim\rB_{n+2}({\les}n)=\brd\rB_{n+2}({\les}n)=n+1$, for every positive integer~$n$.
\end{lemma}

The following corollary is observed in \cite[Th\'eor\`eme~3.3.13]{GillTh}.

\begin{corollary}[Gillibert]\label{C:Pierren+2}
The relation $(\gl^{+n},n,\gl)\to n+2$ holds for each infinite cardinal~$\gl$ and each positive integer~$n$.
\end{corollary}

In particular, this answers in the affirmative the question, raised on \cite[page~285]{EHMR}, whether $(\aleph_4,4,\aleph_0)\to6$.

These methods show that in order to prove the existence of large free sets for set mappings of type a positive integer~$r$ (as defined in \cite[Section~46]{EHMR}), it is sufficient to establish upper bounds for the order-dimension of finite lattices of the form~$\rB_m({\les}r)$. This problem gets somewhat simplified by using the following easy result.

\begin{lemma}\label{L:dim1tor}
The equality $\dim\rB_m({\les}r)=\dim\rB_m(1,r)$ holds, for all integers~$m$ and~$r$ such that $1\leq r<m$.
\end{lemma}

\begin{proof}
As $\rB_m({\les}r)$ contains $\rB_m(1,r)$, the inequality $\dim\rB_m({\les}r)\geq\dim\rB_m(1,r)$ is trivial. Conversely, set $N:=\dim\rB_m(1,r)$ and let~$K$ be a product of chains (any lattice would do) with an order-embedding $\gf\colon\rB_m(1,r)\into K$. We set
 \[
 \gy(X):=\bigvee\famm{\gf(\set{i})}{i\in X}\,,
 \quad\text{for each }X\in\rB_m({\les}r)\,.
 \]
Clearly~$\gy$ is isotone. Let $X,Y\in\rB_m({\les}r)$ such that $\gy(X)\les\gy(Y)$, we must prove that $X\subseteq Y$. We may assume that $Y\neq m$. In particular, $\card Y\leq r$, and thus, by part of the proof of Proposition~\ref{P:FreeB(k,m)}, we can write $Y=\bigcap\famm{Y_j}{j<s}$, for a positive integer~$s$ and $Y_0,\dots,Y_{s-1}\in[m]^r$. For each $i\in X$ and each $j<s$,
 \[
 \gf(\set{i})\leq\gy(X)\leq\gy(Y)\leq\gy(Y_j)\leq\gf(Y_j)\,,
 \]
thus, as~$\gf$ is an order-embedding, $i\in Y_j$. As this holds for all possible choices of~$i$ and~$j$, we obtain that $X\subseteq Y$. Therefore, $\gy$ is an order-embedding, and so $\rB_m({\les}r)$ order-embeds into the same product of chains as $\rB_m(1,r)$ does.
\end{proof}

Getting estimates of the order-dimension of $\rB_m(1,r)$ has given rise to a great deal of work, starting with Dushnik~\cite{Dush50}. Further refinements can be found, for example, in Kierstead~\cite{Kier96,Kier99}. We shall illustrate how large free sets can be obtained from small order-dimensions in Proposition~\ref{P:LargeFreeSets}.

We shall use Dushnik's work~\cite{Dush50}. For integers $m$, $r$ with $1\leq r\leq m$, Dushnik denotes by $N(m,r)$ the minimal number~$N$ such that there exists a set~$\cS$ of~$N$ linear orderings of~$m$ such that for each $A\in[m]^r$ and each $a\in A$, there exists $S\in\cS$ such that $(x,a)\in S$ for each $x\in A$. Then Dushnik establishes in \cite[Theorem~III]{Dush50} that
 \begin{equation}\label{Eq:dim=N(mk)}
 \dim\rB_m(1,r)=N(m,r+1)\,,\quad\text{for all integers }m,\,r\text{ such that }
 1<r<m\,.
 \end{equation}
In order to get estimates of~$N(m,r+1)$, we shall use Hajnal's work quoted in Spencer's paper~\cite{Spen71}. For integers~$n$ and~$r$ such that $1\leq r\leq n$, Spencer denotes by $M(n,r)$ the maximal cardinality of an ``$r$-scrambling'' family of subsets of~$n$, and he establishes on \cite[Lemma, p.~351]{Spen71} the inequality
 \begin{equation}\label{Eq:LBforM}
 M(n,r)\geq\intgr{\frac{1}{2}(1-2^{-r})^{-n/r}}\,,
 \end{equation}
where $\intgr{x}$ denotes the largest integer below any real number~$x$. Furthermore, Spencer establishes on \cite[page~351]{Spen71} the inequality
 \begin{equation}\label{Eq:NfromM}
 N\bigl(2^{M(n,r)},r+1\bigr)\leq n\,,\quad\text{for all integers }n,\,r
 \text{ such that }1\leq r<n\,.
 \end{equation}
Now let $n$ and~$r$ be integers such that $1<r<n$, and set
 \begin{equation}\label{Eq:defnEnr}
 m:=2^{\intgr{\frac{1}{2}(1-2^{-r})^{-n/r}}}\,. 
 \end{equation}
It follows from~\eqref{Eq:LBforM} that $2^{M(n,r)}\geq m$, thus, by the isotonicity of the function~$N$, it follows from~\eqref{Eq:NfromM} that $N(m,r+1)\leq N\bigl(2^{M(n,r)},r+1\bigr)\leq n$, and thus, by~\eqref{Eq:dim=N(mk)}, $\dim\rB_m(1,r)\leq n$. Therefore, by Lemma~\ref{L:dim1tor}, $\dim\rB_m({\les}r)\leq n$ as well, and therefore, by Proposition~\ref{P:UppBd}, $\kur\rB_m({\les}r)\leq n$. Now an immediate application of Proposition~\ref{P:FreeB(k,m)} yields the relation $(\gk^{+(n-1)},r,\gk)\rightarrow m$, for each infinite cardinal~$\gk$. This completes the proof of the following result.

\begin{proposition}\label{P:LargeFreeSets}
Let $n$ and~$r$ be positive integers with $2\leq r<n$. Then the relation
$(\gl^{+(n-1)},r,\gl)\rightarrow 2^{\intgr{\frac{1}{2}(1-2^{-r})^{-n/r}}}$ holds for each infinite cardinal~$\gl$.
\end{proposition}

Denote by $\lg x$ and $\log x$ the base two logarithm, resp. the natural logarithm of a positive real number~$x$. Set $E(n,r):=2^{\intgr{\frac{1}{2}(1-2^{-r})^{-n/r}}}$, for all integers~$n$, $r$ such that $2\leq r<n$. Elementary calculations yield the asymptotic estimate
 \[
 \lg\lg E(n,r)\sim \frac{n}{r2^r\log 2}\,,\quad
 \text{as }n\gg r\gg 0\,.
 \]
We illustrate on Table~\ref{Ta:SmallE} the behavior of the function~$E$ on $r:=4$ and (relatively) small values of~$n$. The given values of~$n$ are minimal for the corresponding values of~$E(n,4)$: for example, $E(171,4)=128$ while $E(172,4)=256$. The value $n:=172$ is the first one for which Proposition~\ref{P:LargeFreeSets} gives a nontrivial large free set result.

\begin{table}
\begin{tabular}{|| l || r | r | r | r | r | r | r | r | r ||}
\hline
$n$ & 172 & 180 & 186 & 192 & 197 & 202 & 207 & 211 & 215\\ \hline
$E(n,4)$ & 256 & 512 & 1,024 & 2,048 & 4,096 & 8,192 & 16,384 & 32,768 & 65,536\\ \hline
\end{tabular}
\caption{\tvi Small values of $E(n,4)$}\label{Ta:SmallE}
\end{table}
In particular, we deduce from Proposition~\ref{P:LargeFreeSets} the following relations:
 \begin{align*}
 (\aleph_{210},4,\aleph_0)&\rightarrow 32{,}768\,;\\
 (\aleph_{214},4,\aleph_0)&\rightarrow 65{,}536\,.
 \end{align*}
and so on. Smaller values of~$n$ also give nontrivial large free sets, however, for those values the bounds that we shall discuss now are better.

We recall the following estimate of the order-dimension of~$\rB_m(1,r)$. This estimate is first stated explicitly in Brightwell \emph{et al.}~\cite[Proposition~1.5]{BKKT94}, where it is attributed to F\"uredi and Kahn~\cite{FuKa86}. However, the estimate is not stated explicitly in the latter paper, although the proof is an easy modification of the proof of that paper's~\cite[Proposition~2.3]{FuKa86}.

\begin{proposition}[F\"uredi and Kahn]\label{P:FuKa}
Let $r$ and~$m$ be integers such that $1\leq r<m$ and let~$d$ be a positive integer satisfying the inequality
 \[
 m\binom{m-1}{r}\Bigl(\frac{r}{r+1}\Bigr)^d<1\,.
 \]
Then $\dim\rB_m(1,r)\leq d$.
\end{proposition}

If $\dim\rB_m(1,r)\leq d$, then, by Lemma~\ref{L:dim1tor}, $\dim\rB_m({\les}r)\leq d$, thus, by Proposition~\ref{P:UppBd}, $\kur\rB_m({\les}r)\leq d$, and thus, by Proposition~\ref{P:FreeB(k,m)}, $(\aleph_{d-1},r,\aleph_0)\rightarrow m$. In particular, we obtain the relation
 \[
 (\aleph_{109},4,\aleph_0)\rightarrow 257\,,
 \]
which is better than the relation $(\aleph_{171},4,\aleph_0)\rightarrow 256$ given by Table~\ref{Ta:SmallE}. However, for $r=4$ and $n\geq211$ the size of the free set given by Proposition~\ref{P:LargeFreeSets} becomes larger again.

We can even obtain further large free sets results by using Dushnik's original exact estimate of the order-dimension of a truncated cube under certain conditions~\cite{Dush50}.

\begin{proposition}[Dushnik]\label{P:Dush50}
Let $m$, $j$, and $k$ be positive integers such that $m\geq4$, $2\leq j\leq\intgr{\sqrt{m}}$, and
 \[
 \intgr{\frac{m+j^2-j}{j}}\leq k<\intgr{\frac{m+(j-1)^2-j+1}{j-1}}\,.
 \]
Then $\dim\rB_m(1,k-1)=m-j+1$.
\end{proposition}

In the context of Proposition~\ref{P:Dush50}, it follows from Lemma~\ref{L:dim1tor} that\linebreak $\dim\rB_m({\les}k-1)=m-j+1$, thus, by Proposition~\ref{P:UppBd}, $\kur\rB_m({\les}k-1)\leq m-j+1$, and thus, by Proposition~\ref{P:FreeB(k,m)}, the relation $(\aleph_{m-j},k-1,\aleph_0)\rightarrow m$ holds.

Interesting values are obtained for $k=5$ and $m\in\set{10,11}$, giving respectively the relations
 \[
 (\aleph_7,4,\aleph_0)\rightarrow 10\quad\text{and}\quad
 (\aleph_8,4,\aleph_0)\rightarrow 11\,.
 \]
For $k=6$ and $m\in\set{12,13,14}$, we obtain, respectively,
 \[
 (\aleph_9,5,\aleph_0)\rightarrow 12\,,\quad
 (\aleph_{10},5,\aleph_0)\rightarrow 13\,,\quad
 (\aleph_{11},5,\aleph_0)\rightarrow 14\,.
 \]
Of course, the ``origin'' $\aleph_0$ can be changed, in any of these relations, to any infinite cardinal, so, for example, the relation $(\gl^{+9},5,\gl)\rightarrow 12$ holds for any infinite cardinal~$\gl$.

\section{Discussion}
Our work \cite{Larder} makes a heavy use of Kuratowski indexes of finite lattices, in particular in order to evaluate ``critical points'' between quasivarieties of algebraic structures. Now evaluating the Kuratowski index of a finite poset may be a hard problem, even in the case of easily describable finite lattices. For example, consider the finite lattices~$P$ and~$Q$ represented in Figure~\ref{Fig:lattices}.
\begin{figure}[htp]
\includegraphics{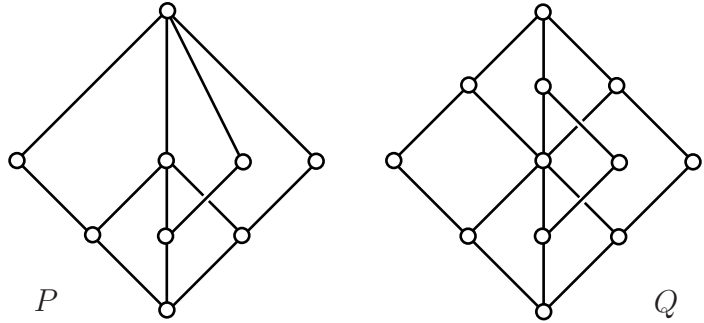}
\caption{Two lattices of breadth two and order-dimension three}\label{Fig:lattices}
\end{figure}
Both~$P$ and~$Q$ have breadth two and order-dimension three. Furthermore, $P$ embeds, as a poset, into~$Q$, thus, by Proposition~\ref{P:KurIndIncr}, $\kur(P)\leq\kur(Q)$. Therefore, by Propositions~\ref{P:UppBd} and~\ref{P:kurgewbr}, we obtain the inequalities
 \[
 2\leq\kur(P)\leq\kur(Q)\leq 3\,.
 \]
We do not know which one of those inequalities can be strengthened to an equality. The statement $\kur(P)=2$ is equivalent to
\begin{quote}
For every infinite cardinal~$\gl$ and every $F\colon[\gl^+]^{<\go}\to[\gl^+]^{<\gl}$, there are distinct $\xi_0,\xi_1,\xi_2,\eta_0,\eta_1,\eta_2<\gl$ such that $\xi_i\notin F(\set{\xi_j,\eta_j})$, $\eta_i\notin F(\set{\xi_j,\eta_j})$, and $\eta_i\notin F(\set{\xi_0,\xi_1,\xi_2})$ for all $i\neq j$ in $\set{0,1,2}$,
\end{quote}
while the statement $\kur(Q)=2$ is equivalent to the apparently stronger statement
\begin{quote}
For every infinite cardinal~$\gl$ and every $F\colon[\gl^+]^{<\go}\to[\gl^+]^{<\gl}$, there are distinct $\xi_0,\xi_1,\xi_2,\eta_0,\eta_1,\eta_2<\gl$ such that $\xi_i\notin F(\set{\xi_j,\eta_j})$ and $\eta_i\notin F(\set{\xi_0,\xi_1,\xi_2,\eta_j})$ for all $i\neq j$ in $\set{0,1,2}$.
\end{quote}

Such results could be relevant in further improving ``large free sets'' results such as Corollary~\ref{C:Pierren+2} and Proposition~\ref{P:LargeFreeSets}. In particular, the gap between~$E(n,4)$ (which gives a positive large free set result in Proposition~\ref{P:LargeFreeSets}) and~$t_n$ (which gives Komj\'ath and Shelah's negative large free set result) looks huge. Can the bound~$E(n,4)$ be improved in Proposition~\ref{P:LargeFreeSets}? On the other hand, we do not even know whether the relation $(\aleph_4,4,\aleph_0)\rightarrow7$ holds.


\begin{thebibliography}{99}
\bibitem{Baker}
K.\,A. Baker, ``Dimension, join-independence, and breadth in partially ordered sets'', honors thesis 1961 (unpublished).

\bibitem{BKKT94}
G.\,R. Brightwell, H.\,A. Kierstead, A.\,V. Kostochka, and W.\,T. Trotter,
\emph{The dimension of suborders of the Boolean lattice}, Order~\textbf{11} (1994), 127--134.

\bibitem{DePa}
K.\,J. Devlin and J.\,B. Paris,
\emph{More on the free subset problem}, Ann. Math. Logic~\textbf{5} (1972-1973), 327--336.

\bibitem{Dito84}
S.\,Z. Ditor,
\emph{Cardinality questions concerning semilattices of finite breadth},
Discrete Math. \textbf{48} (1984), 47--59.

\bibitem{Dush50}
B. Dushnik,
\emph{Concerning a certain set of arrangements}, 
Proc. Amer. Math. Soc.~\textbf{1}, (1950). 788--796.

\bibitem{EHMR}
P. Erd\H{o}s, A. Hajnal, A. M\'at\'e, and R. Rado,
``Combinatorial Set Theory: Partition Relations for Cardinals''.
Studies in Logic and the Foundations of Mathematics~\textbf{106}. North-Holland Publishing Co., Amsterdam, 1984. 347~p. ISBN: 0-444-86157-2

\bibitem{FuKa86}
Z. F\"uredi and J. Kahn,
\emph{On the dimensions of ordered sets of bounded degree},
Order~\textbf{3}, no.~1 (1986), 15--20.

\bibitem{GillTh}
P. Gillibert, ``Points critiques de couples de vari\'et\'es d'alg\`ebres'', Doc\-to\-rat de l'U\-ni\-ver\-si\-t\'e de Caen, December~8, 2008. Available online at \texttt{http://tel.archives-ouvertes.fr/tel-00345793}\,.

\bibitem{Gill1}
P. Gillibert, \emph{Critical points of pairs of varieties of algebras}, Internat. J. Algebra Comput.~\textbf{19}, no.~1 (2009), 1--40.

\bibitem{Gill2}
P. Gillibert,
\emph{Critical points between varieties generated by subspace lattices of vector spaces}, J. Pure Appl. Algebra~\textbf{214} (2010), 1306--1318.

\bibitem{Larder}
P. Gillibert and F. Wehrung, \emph{{}From objects to diagrams for ranges of functors}, preprint 2010. Available online at \texttt{http://hal.archives-ouvertes.fr/hal-00462941}\,.

\bibitem{HaMa75}
A. Hajnal and A. M\'at\'e,
\emph{Set mappings, partitions, and chromatic numbers}. Logic Colloquium '73 (Bristol, 1973), p. 347--379. Studies in Logic and the Foundations of Mathematics, Vol.~\textbf{80}, North-Holland, Amsterdam, 1975.

\bibitem{Kier96}
H.\,A. Kierstead,
\emph{On the order-dimension of $1$-sets versus $k$-sets},
J. Combin. Theory Ser. A~\textbf{73} (1996), 219--228.

\bibitem{Kier99}
H.\,A. Kierstead,
\emph{The dimension of two levels of the Boolean lattice},
Discrete Math.~\textbf{201} (1999), 141--155.

\bibitem{Koep84}
P. Koepke,
\emph{The consistency strength of the free-subset property for $\go\sb\go$}, J. Symbolic Logic~\textbf{49} (1984), no.~4, 1198--1204.

\bibitem{Koep89}
P. Koepke,
\emph{On the free subset property at singular cardinals},
Arch. Math. Logic~\textbf{28} (1989), no.~1, 43--55.

\bibitem{KoSh00}
P. Komj\'ath and S. Shelah,
\emph{Two consistency results on set mappings}, J. Symbolic Logic~\textbf{65} (2000), 333--338.

\bibitem{Kura51}
C. Kuratowski,
\emph{Sur une caract\'erisation des alephs},
Fund. Math. \textbf{38} (1951), 14--17.

\bibitem{Laza36}
D. L\'az\'ar,
\emph{On a problem in the theory of aggregates},
Compositio Math.~\textbf{3} (1936), 304--304.

\bibitem{Spen71}
J. Spencer, \emph{Minimal scrambling sets of simple orders}, Acta Math. Acad. Sci. Hungar.~\textbf{22} (1971/72), 349--353.

\end{thebibliography}
\end{document}